\documentclass[12pt]{article}

\usepackage{fullpage}
\usepackage{amssymb}
\usepackage{amsmath}
\usepackage{amsthm}
\usepackage{url}
\usepackage{hyperref}
\usepackage{graphicx}
\usepackage{multicol}
\usepackage{caption}
\usepackage{subcaption}

\newtheorem{theorem}{Theorem}
\newtheorem{corollary}{Corollary}

\newtheorem{lemma}{Lemma}

\newtheorem{definition}{Definition}

\begin{document}

\title{Minimal number of points on a grid forming line
segments of equal length}

\date{August 25, 2016\\Latest update: May 12, 2022}
\author{Chai Wah Wu\\IBM Research AI\\IBM T. J. Watson Research Center\\P. O. Box 218, Yorktown Heights, New York 10598, USA\\e-mail: cwwu@us.ibm.com}
\maketitle

\begin{abstract}
We consider the minimal number of points on a regular grid on the plane that generates $n$ line segments of points of exactly length $k$. We illustrate how this is related to the $n$-queens problem on the toroidal chessboard and show that this number is upper bounded by $kn/3$ and approaches $kn/4$ as $n\rightarrow\infty$ when $k+1$ is coprime with $6$ or when $k$ is large.
\end{abstract}

\section{Introduction}
We consider points on a regular grid on the plane which form horizontal, vertical or diagonal line segments of exactly  $k$ points \footnote{We use the convention that an isolated point corresponds to $4$ line segments of length $1$; a horizontal, a vertical and 2 diagonal line segments.}.   For example, the set of 12 points in Fig. \ref{fig:one} form many line segments and form exactly 3 (overlapping) line segments of length $5$. Note that since a line segment of length $k$ consists of exactly $k$ points and no more\footnote{This implies that two line segments of the same orientation (horizontal, vertical or diagonal) must not overlap.}, the set of points in Fig. \ref{fig:one} contains 4 line segments of length 2 and does not contain any line segments of length $4$ or of length $3$.   Our motivation for studying this problem is the Bingo-4 problem proposed by Sun et al.  and described in OEIS\cite{OEIS} sequence \href{http://oeis.org/A273916}{A273916} where  the case $k=4$ is considered. This problem
 can be considered a type of orchard-planting problem \cite{burr:orchard:1974} restricted to a grid.  
 
Let $a_k(n)$ denote the minimal number of points needed to form $n$ line segments of length $k$. Fig. \ref{fig:one} shows that $a_5(3) = 12$ as any constellation of $11$ points will not generate $3$ segments of length $5$. Note that the constellation of points achieving $a_k(n)$ is typically not unique. Finding the exact value of $a_k(n)$ appears to be difficult and currently not feasible for large $n$. The purpose of this note is to provide an analysis on the asymptotic behavior of $a_k(n)$.

\begin{figure}[htbp]
\centerline{\includegraphics[width=3in]{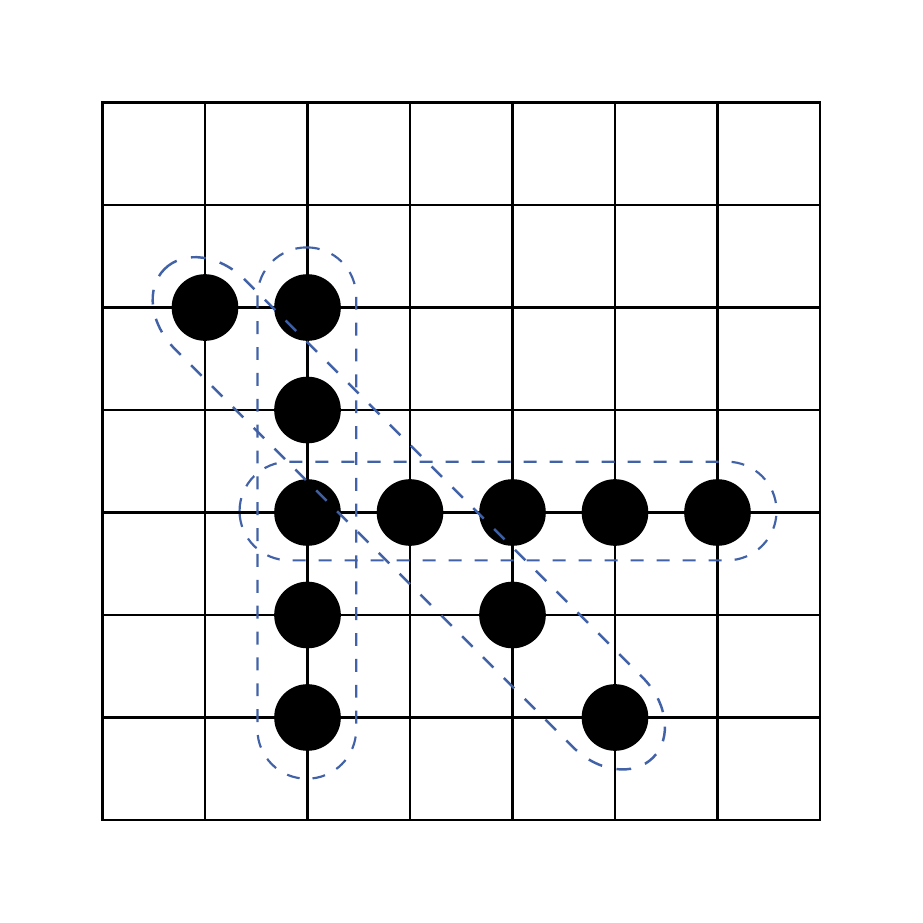}}
\caption{A constellation of 12 points on a grid. Among the line segments formed by the points there are 3 (overlapping) line segments of length 5.}\label{fig:one}
\end{figure}

\section{Bounds and asymptotic behavior of $a_k(n)$}
It is easy to show that $a_k(1) = k$, $a_k(2) = 2k-1$ and $a_k(3) = 3(k-1)$ as $2$ line segments overlap in at most one point and $3$ line segments overlaps in at most $3$ points, as illustrated in Fig. \ref{fig:smalln} for $k=5$. Note that $a_k(3)$ can be obtained with points forming a right isosceles triangle. 
\begin{lemma}[Fekete's subadditive Lemma \cite{fekete:subadditive:1923}]
If the sequence $a(n)$ is subadditive, i.e. $a(n+m) \leq a(n)+a(m)$, then $\lim_{n\rightarrow\infty}\frac{a_n}{n}$ exists and is equal to $\inf_n \frac{a_n}{n}$.
\label{lem:fekete}
\end{lemma}

\begin{figure}[htbp]
     \centering
     \begin{subfigure}[b]{0.3\textwidth}
         \centering
         \includegraphics[width=\textwidth]{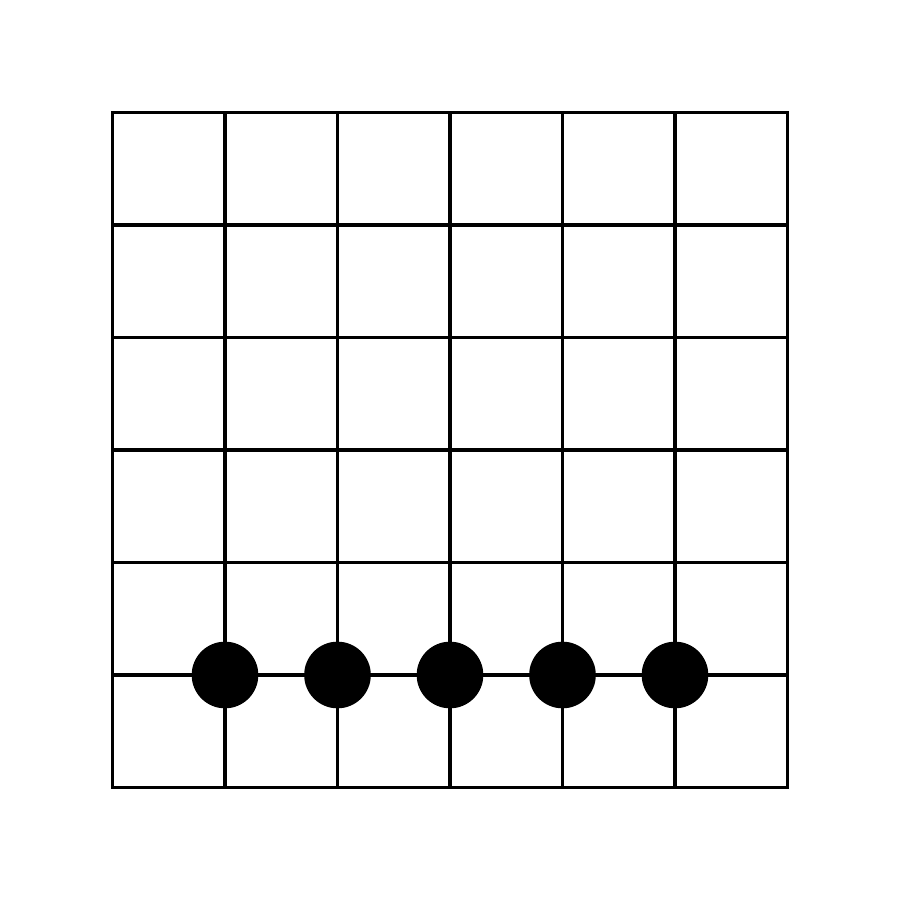}
         \caption{$a_5(1) = 5$}
         \label{fig:fig_n=1}
     \end{subfigure}
     \hfill
     \begin{subfigure}[b]{0.3\textwidth}
         \centering
         \includegraphics[width=\textwidth]{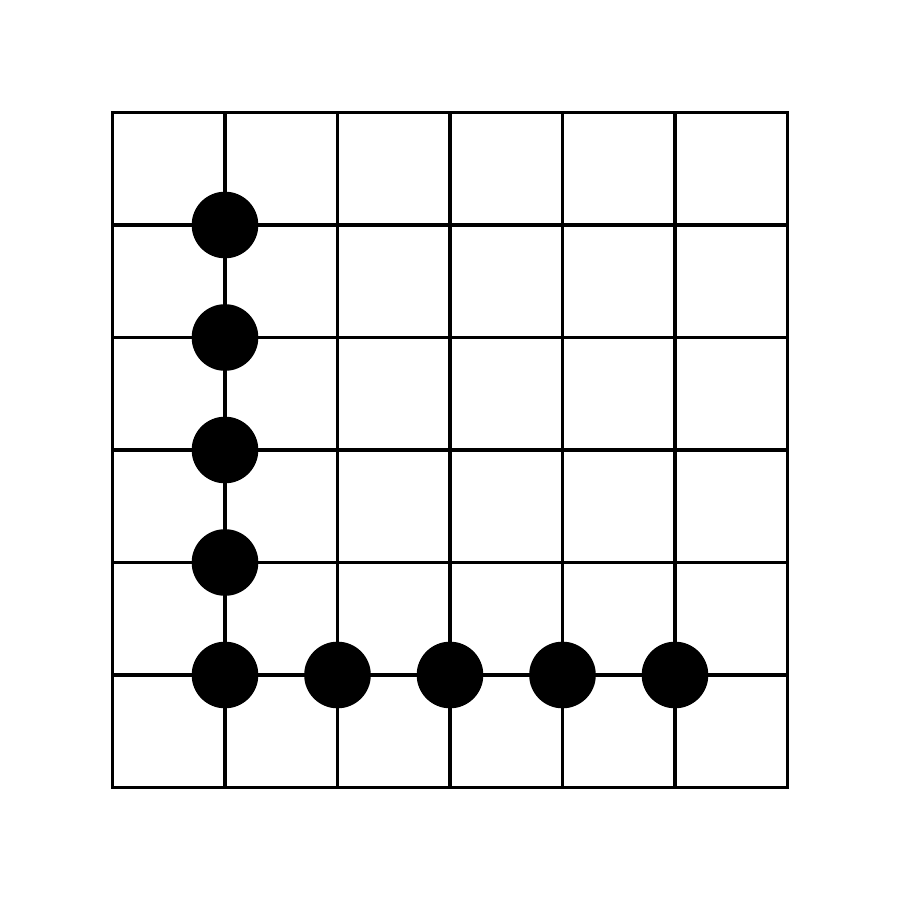}
         \caption{$a_5(2) = 9$}
         \label{fig:fig_n=2}
     \end{subfigure}
     \hfill
     \begin{subfigure}[b]{0.3\textwidth}
         \centering
         \includegraphics[width=\textwidth]{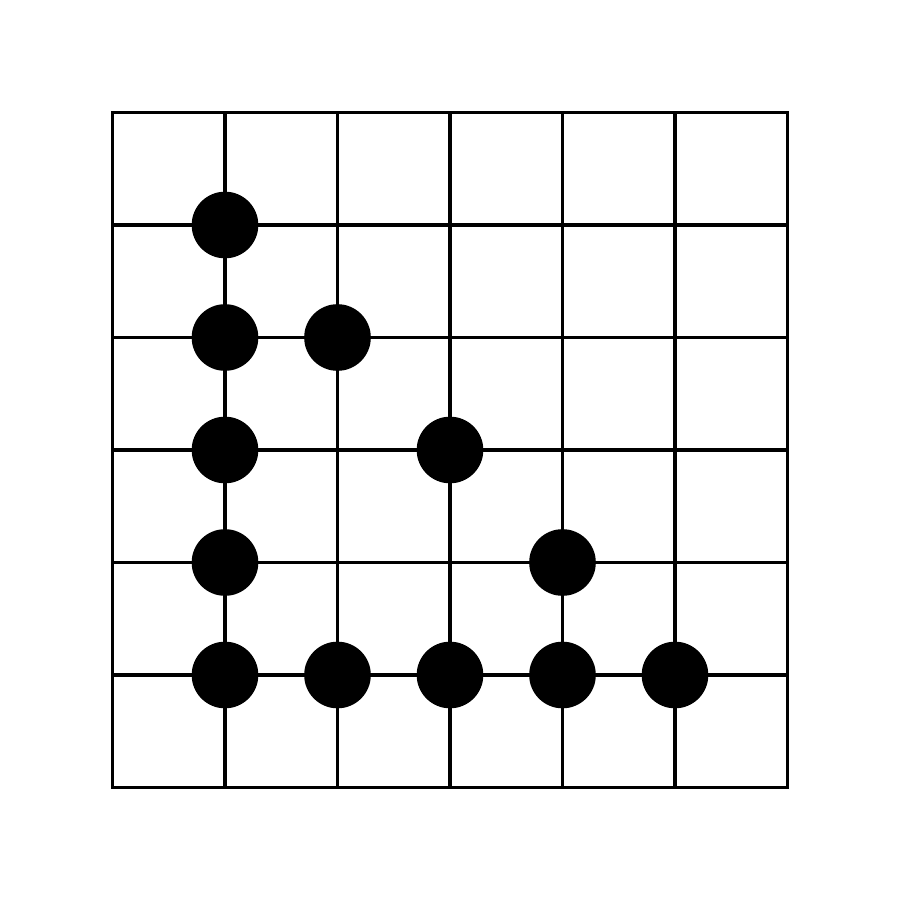}
         \caption{$a_5(3) = 12$}
         \label{fig:fig_n=3}
     \end{subfigure}
        \caption{Sets of points illustrating $a_k(n)$ for $n=1,2,3$.}
        \label{fig:smalln}
\end{figure}

\begin{theorem}
For all $k$, $a_k(n)$ is subadditive, and $f(k) = \lim_{n\rightarrow\infty}\frac{a_k(n)}{n} $ exists and satisfies $\frac{k}{4}\leq f(k) \leq \frac{k}{3}$.
\label{thm:bound}
\end{theorem}
\begin{proof}
Since each line segment takes $k$ points and each point can be part of at most $4$ line segments (horizontal, vertical or diagonal), $a_k(n) \geq \frac{kn}{4}$.
Since the set of points for $a_k(n)$ and $a_k(m)$ separated apart leads to $m+n$ line segments of length $k$, it is clear that $a_k(n)$ is subadditive. Lemma \ref{lem:fekete} implies that $f(k)$ exists and is equal to $\inf_n \frac{a_k(n)}{n}$.
Consider a $k$ by $m$ rectangular array of points with $k\leq m$. There are $m$ vertical line segments and $m-k+1$ diagonal
line segments of each orientation and thus there are $3m-2k+2$ length $k$ line segments. This shows that $a_k(3m-2k+2)\leq km$ which implies that
$\frac{k}{4}\leq f(k) \leq \frac{k}{3}$. 
\end{proof}

\section{Constellations where each point is part of 4 different line segments}
The upper bound $\frac{k}{3}$ on $f(k)$ in Theorem \ref{thm:bound} shows that for large $n$ we can construct a constellation of $n$ points such that most points are part of $3$ different line segments. Is it possible to construct a constellation such that most points are part of $4$ different line segments (a horizontal, a vertical and two diagonal line segments) and thus achieve the lower bound $\frac{k}{4}$? The case $k=1$ is simple. Since $a_1(4n) = n$ as exhibited by the constellation of $n$ isolated points, this implies that $f(1) = \frac{1}{4}$.

Let $\sigma\in S_{k+1}$ be a permutation on the integers $\{0,1,\cdots , k\}$. Consider a $k+1$ by $k+1$ square grid and place a point on each position $(i,j)$ except when it is of the form $(i,\sigma(i))$. It is clear that tiling this grid on the plane results in a constellation where every point is part of a horizontal and a vertical line segment of length $k$. The shear maps
$(i,j) \rightarrow (i,i+j)$ and $(i,j)\rightarrow (i,i-j)$ map the two diagonal line segments to a vertical line segment. Thus 
in order to also have every point be part of two diagonal line segments of exactly $k$ points, we want
$\{i+\sigma(i) \mod (k+1)\}$ and $\{i-\sigma(i) \mod (k+1)\}$ to be permutations of  $\{0,1,\cdots , k\}$ as well.  If this is the case, consider a $N$ by $N$ subgrid of this tiling and let $n$ be the number of points in this subgrid.  Except for points near the edges which is on the order of $kN \propto k\sqrt{n}$, all points belong to $4$ line segments of length $k$.  Thus we have proved the following:
\begin{theorem} \label{thm:perm}
If there is a  permutation $\sigma$ of the numbers $\{0,1,\cdots ,k\}$ such that
$\sigma_1 = \{i+\sigma(i) \mod (k+1)\}$ and $\sigma_2 = \{i-\sigma(i) \mod (k+1)\}$ are both permutations, then $f(k) = \frac{k}{4}$. In particular, $\frac{a_k(n)}{n}$ converges to $f(k)$ on the order of 
$O\left(\frac{1}{\sqrt{n}}\right)$. 
\end{theorem}

If $\sigma$ satisfies the conditions of Theorem \ref{thm:perm}, then so does $\sigma^{-1}$.  For a fixed integer $m$, the permutation
$\sigma(i) + m \mod (k+1)$ also satisfies these conditions. We will use this to partition the set of admissible permutations into equivalent classes. More specifically,
\begin{definition}
Let $S_{k+1}$ be the set of permutations on $\{0,1,\cdots , k\}$.  $T_{k+1}\subset S_{k+1}$ is defined as the set of permutations $\sigma$ such that  $\{i+\sigma(i)\mod (k+1)\}$ and $\{i-\sigma(i)\mod (k+1)\}$ are in $S_{k+1}$.  The equivalence relation $\sim$ is defined on $T_{k+1}$ as follows. If $\sigma, \tau \in T_{k+1}$, then $\sigma \sim \tau$ if $\tau = \sigma^{-1}$ or
there exist an integer $m$ such that $\sigma(i) = \tau(i) + m \mod (k+1)$ for all $i$.
\end{definition}

Thus Theorem \ref{thm:perm} implies that if $T_{k+1}\neq \varnothing$, then $f(k) = \frac{k}{4}$.

\section{Modular $n$-queens problem}
In this section we show that the above constellation is related to an $n$-queens problem on a toroidal chessboard.
The $n$-queens problem asks whether $n$ nonattacking queens can be placed on an $n$ by $n$ chessboard.  The answer is yes and is first shown by Pauls \cite{pauls:nqueens:1874,bell:nqueens:2009}.
Next consider a toroidal $n$ by $n$ chessboard, where the top edge is connected to the bottom edge and the left edge is connected to the right edge. The corresponding $n$-queens problem is called a {\em modular} $n$-queens problem.
For the $k+1$ by $k+1$ square grid above, if we put a queen on each position $(i,\sigma(i))$, then 
it is easy to see that $\sigma\in T_{k+1}$ if and only if it provides a solution to the modular $(k+1)$-queens problem.  
For instance, for $k=4$, consider the permutation $\sigma = (0,2,4,1,3)$. Figure \ref{fig:k=4} shows a $5$ by $5$ grid where all the points are part of $4$ line segments if the grid tiles the plane (or equivalently, the grid lives on a torus). This means that each point in the center of a finite tiling are part of $4$ line segments. If we put a queen on each of the $5$ empty locations, we obtain a solution to the modular $5$-queens problem.

\begin{figure}[htbp]
\centerline{\includegraphics[width=3in]{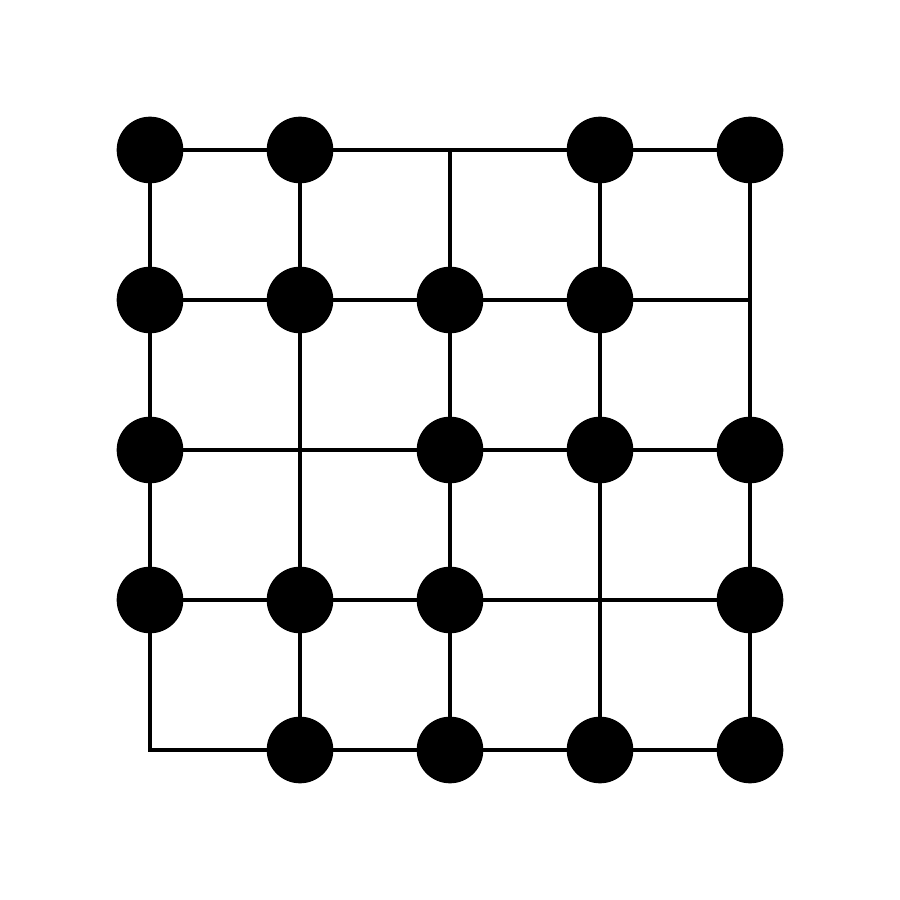}}
\caption{Points where the empty locations are of coordinates $(i,\sigma(i))$. Putting a queen at each empty location results in a solution to the modular $5$-queen problem.}\label{fig:k=4}
\end{figure}

P\'{o}lya \cite{polya:nqueens:1918} showed that a solution to the modular $n$-queens problem exists if and only if $n$ is coprime with $6$.
Thus P\'{o}lya's result is equivalent to the following:
\begin{theorem}\label{thm:polyaT}
$T_{k+1} \neq \varnothing$ if and only if $k+1$ is coprime with $6$.
\end{theorem}

\begin{corollary}\label{cor:one}
If $k+1$ is coprime with $6$, then $f(k) = \frac{k}{4}$.
\end{corollary}

Monsky \cite{monsky:nqueens:1989} shows that $n-2$ nonattacking queens can be placed on an $n$ by $n$ toroidal chess board and $n-1$ queens can be placed if
$n$ is not divisible by $3$ or $4$.
This implies the following which shows that for $k$ large, $f(k)$ approaches the lower bound $\frac{k}{4}$:

\begin{theorem}\label{cor:upperbound}
$f(k) \leq \frac{k(k+1)+2}{4(k-1)}$. If $k+1$ is not divisible by $3$ or $4$, then
$f(k) \leq \frac{k(k+1)+1}{4k}$.
\end{theorem}

\begin{proof}
Consider a $k+1$ by $k+1$ array with $k+1-r$ nonattacking queens. By placing a point only on the locations where there are no queens we obtain a constellation with $(k+1)^2-(k+1-r)$ points.
Each point then is part of $4$ line segments of length $k$.  Thus  when this array is tiled, we get for a large number of points a ratio $\frac{a_k(n)}{n}$ approaching
$\frac{(k+1)^2-(k+1-r)}{4(k+1-r)} = \frac{k(k+1)+r}{4(k+1-r)}$.  The conclusion follows by setting $r =1 $ or $r= 2$. 
\end{proof}

\begin{corollary}
$\lim_{k\rightarrow \infty} \frac{f(k)}{k} = \frac{1}{4}$.
\end{corollary}

\subsection{Lattice construction}
As in the $n$-queens problem, we can construct permutations in $T_{k+1}$ via a lattice construction.

\begin{definition}
Given two vectors $v_1$ and $v_2$, the lattice construction is defined as a constellation of points such that 
a point is on the grid if and only if the point is not a linear combination of $v_1$ and $v_2$.
\end{definition}
For instance with the lattice points generated by the vectors $(1,2)$ and $(2,-1)$, the set of points with $N=15$ is shown in Fig. \ref{fig:two}.
In particular, this configuration shows that $f(4) = 1$.

\begin{figure}[htbp]
\centerline{\includegraphics[width=5in]{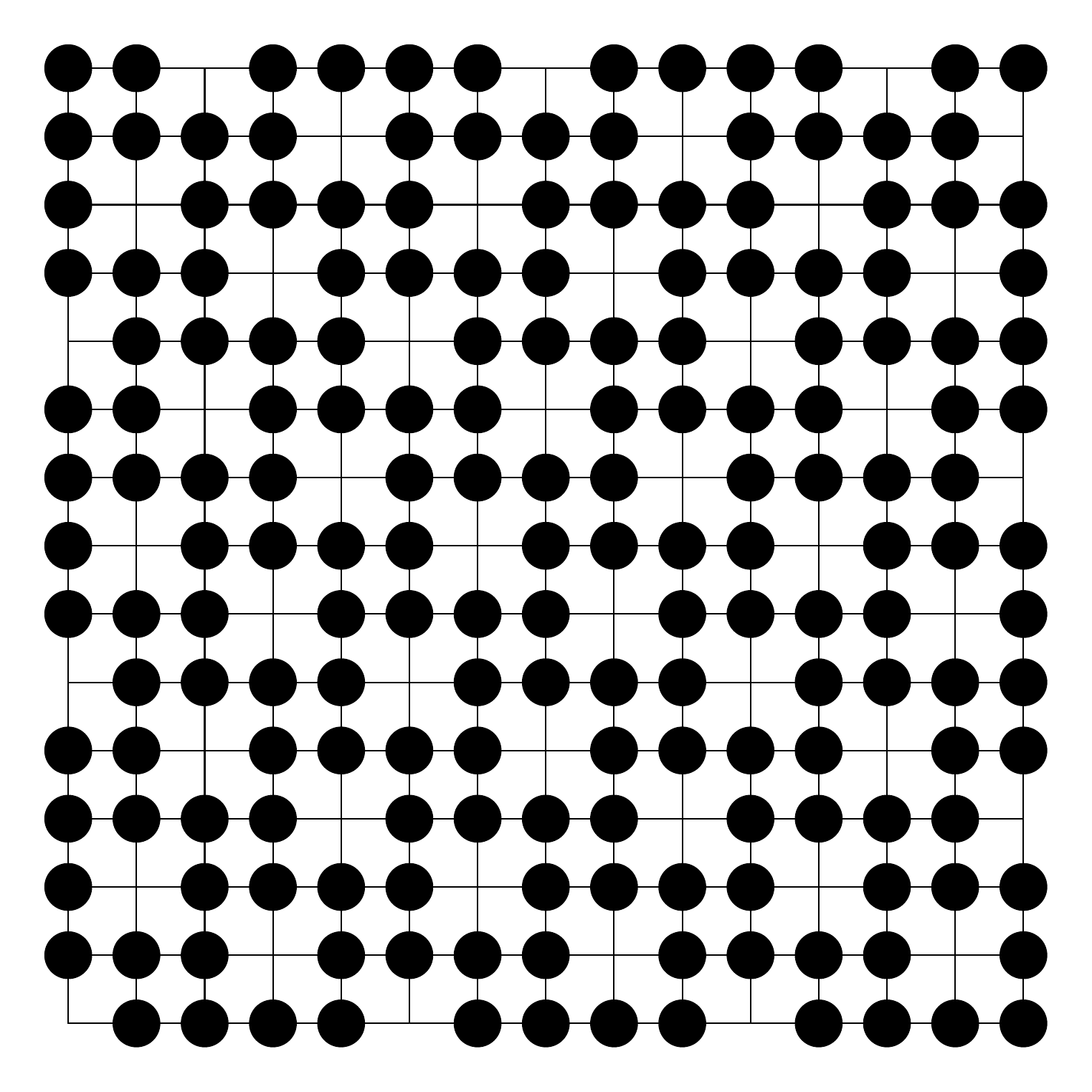}}
\caption{A lattice constellation. Points in the center of the grid are part of $4$ different patterns, showing that $\frac{a_{4}(n)}{n} \rightarrow 1$ as $n\rightarrow \infty$.}\label{fig:two}
\end{figure}

The following result appears to be well-known \cite{bell:nqueens:2009}, but we include it here for completeness.
\begin{theorem} \label{thm:coprime}
If there exists $1 <m < k$ such that $m-1$, $m $ and $m+1$ are all coprime with $k+1$, then the lattice construction with $v_1 = (1,m)$ and $v_2 = (0,k+1)$ corresponds to a permutation $\sigma$ in $T_{k+1}$.
\end{theorem}

\begin{proof}
Consider the lattice construction generated by $(1,m)$ and $(0,k+1)$. If $m$ is coprime with $k+1$, then
$(m,2m,\cdots, (k+1)m) \mod (k+1)$ is a permutation $\sigma$ in $S_{k+1}$ and thus
 we find in a $k+1$ by $k+1$ subarray
empty locations of the form $(i,\sigma(i))$.  
$i+\sigma(i) \equiv (m+1)i \mod (k+1)$ and $\{i+\sigma(i)\mod (k+1)\}$ is again a permutation since $m+1$ and $k+1$ are coprime.
Similarly, $i-\sigma(i) \equiv -(m-1)i \mod (k+1)$ and $\{i-\sigma(i)\mod (k+1)\}$ is a permutation since $m-1$ and $k+1$ are coprime. Thus the conditions of Theorem \ref{thm:perm} are satisfied and the conclusion follows.
\end{proof}

Theorem \ref{thm:coprime} also provides a proof of Corollary \ref{cor:one} since if $k+1$ is coprime with $6$, then $1$, $2$ and $3$ are all coprime with $k+1$ and we can choose $m=2$.  In particular the lattice construction with $v_1 = (1,2)$ and $v_2 = (0,k+1)$ generates a permutation $\sigma$ in $T_{k+1}$. 
Fig. \ref{fig:three} shows the construction for $k= 12$.

\begin{figure}[htbp]
\centerline{\includegraphics[width=5in]{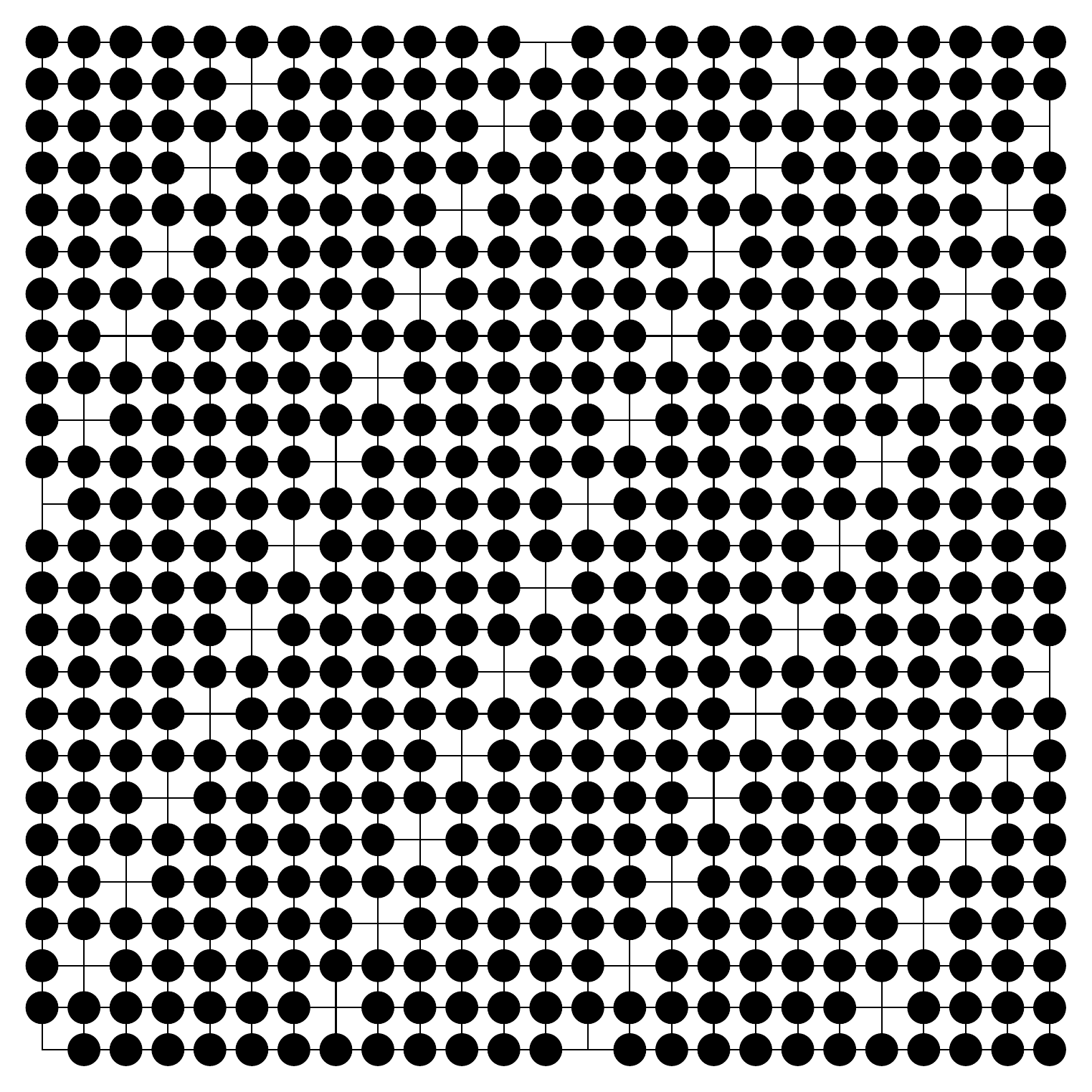}}
\caption{A lattice constellation for $k=12$ generated by vectors $(1,2)$ and $(0,13)$.}\label{fig:three}
\end{figure}

For $k=4$,
there is only one equivalence class $(0,2,4,1,3)$ in $T_{k+1}$ that satisfies the conditions of Theorem \ref{thm:perm}.  For $k=6$, there are two equivalent classes
$(0,2,4,6,1,3,5)$ and $(0,3,6,2,5,1,4)$. For $k=10$, there are $4$ equivalent classes.
In particular, Theorem \ref{thm:coprime} shows that if $k+1 > 4$ is prime, then there are at least $\frac{k-2}{2}$ equivalent classes in $T_{k+1}$. This is because each $2\leq m \leq k-1$ is coprime with $k+1$ and the permutation generated by $m$ is the inverse of the permutation generated by $k-1-m$ which are equivalent\footnote{For general $k$, see \cite{burger:nqueens:2004} for a formula of the number of such permutations.}. It is possible to have more than  $\frac{k-2}{2}$ equivalent classes as there are permutations in $T_{k+1}$ not generated by a lattice.
For $k+1$ coprime with $6$, if $k = 4, 6$ and $10$, all permutations in $T_{k+1}$ are generated by a lattice.  For $k = 12$, there are permutations in $T_{k+1}$ that are not generated by a lattice.  One such example is shown in Fig. \ref{fig:k12}. Such solutions are referred to as {\em nonlinear} solutions \cite{bell:nqueens:2009}. 

\begin{figure}[htbp]
\centerline{\includegraphics[width=5in]{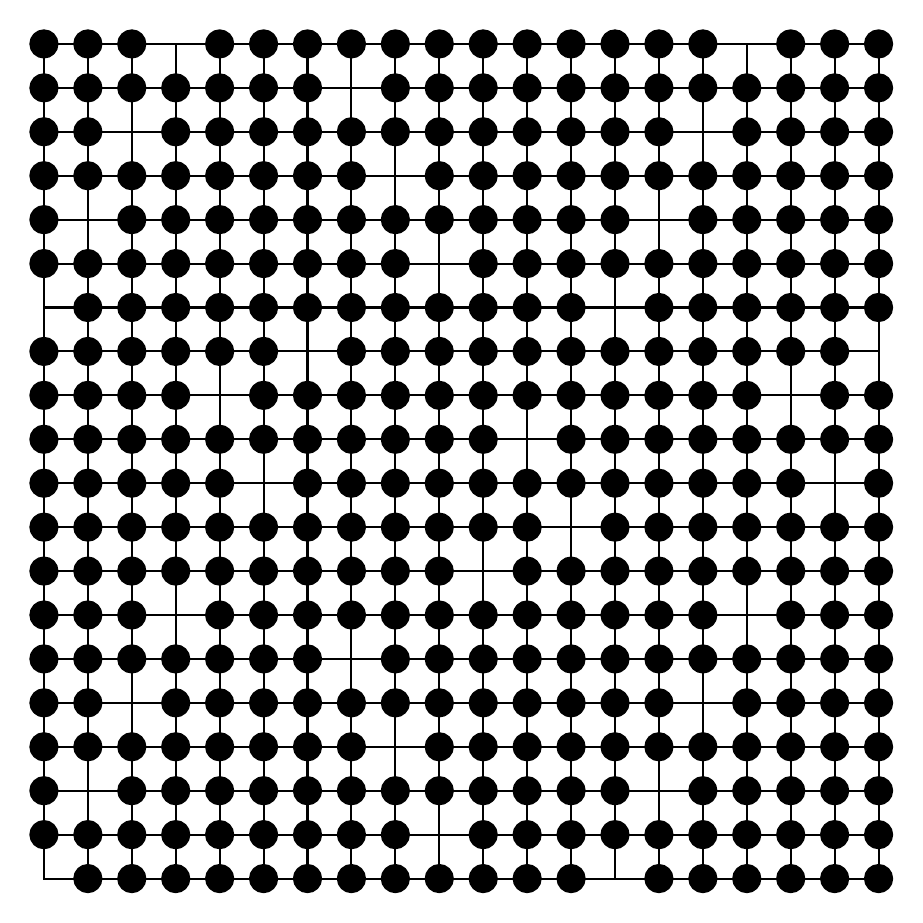}}
\caption{A constellation for $k=12$ not generated by a lattice corresponding to the permutation $(0,2,4,6,11,9,12,5,3,1,7,10,8)$.}\label{fig:k12}
\end{figure}

\section{Conclusions}
We studied the asymptotic behavior of the minimal number of points needed to generate $n$ line segments of length $k$ using a construction based on permutations of
$\{0,1,\cdots, k\}$ with certain properties.  We showed that this construction allows us to create constellations of points where asympotically most points are part of 4 line segments. This construction is equivalent to the modular $(k+1)$-queens problem and thus $f(k)=\frac{k}{4}$ for $k+1$ coprime with $6$. If $k+1$ is even or $k+1$ is divisible by $3$, this construction fails to provide such a constellation.
However, results in the modular $n$-queens problem can still provide an upper bound on $f(k)$ which shows that $\lim_{k\rightarrow\infty}\frac{f(k)}{k} = \frac{1}{4}$. Even though these constructions for the modular $n$-queens problem provide limiting value of $\frac{a_k(n)}{n}$ as $n\rightarrow \infty$, for a fixed $n$ the optimal constellation to achieve $a_k(n)$ can be quite different and difficult to compute (see for example \url{https://oeis.org/A273916/a273916.png}) .

\section{Acknowledgements}
We are indebted to Don Coppersmith for stimulating discussions and for providing many insights during the preparation of this note.

\end{document}